\documentclass{ifacconf}

\usepackage{graphicx}      
\usepackage{natbib}        
\usepackage{mathtools}
\usepackage{amsmath}
\usepackage{amsfonts}
\usepackage{amssymb}
\usepackage{mathrsfs}
\usepackage{float}
\usepackage{subfigure}
\usepackage{color}
\newcounter{coco}
\theoremstyle{definition}

\theoremstyle{plain}
\newtheorem{Teo}{Theorem}
\theoremstyle{plain}
\newtheorem{Cor}[coco]{Corollary}
\theoremstyle{plain}
\newtheorem{Lemma}[coco]{Lemma}
\theoremstyle{plain}

\theoremstyle{example}
\newtheorem{Examp}{Example}
\newenvironment{proof}[1][Proof]{\textbf{#1.} }{\ \rule{0.5em}{0.5em}}

\begin{document}
\begin{frontmatter}

\title{On the Lyapunov Matrix of Linear Delay Difference Equations in
Continuous Time\thanksref{footnoteinfo}} 

\thanks[footnoteinfo]{Project CONACYT 180725}

\author[First]{Emanuel Rocha} 
\author[First]{Sabine Mondi\'e} 
\author[Third]{Michael Di Loreto}

\address[First]{Departamento de Control Autom\'atico, Cinvestav, IPN, M\'exico City (e-mail: erocha,smondie@ctrl.cinvestav.mx).}
\address[Third]{Laboratoire Amp\`{e}re, Universit\'{e} de Lyon, 
   INSA-Lyon, France, (e-mail: michael.di-loreto@insa-lyon.fr)}

\begin{abstract}                
The fundamental matrix and the delay Lyapunov matrix of linear delay difference equations are introduced. Some properties of the Lyapunov matrix, and the jump discontinuities of its derivative are proven, leading to its construction in the case of single delay or commensurate delays. An approximation is proposed for the non-commensurate case.
\end{abstract}


\end{frontmatter}

\section{Introduction}

In Engineering Sciences, motivations to study delay difference equations in continuous time come from sampled-data systems, neutral time-delay systems (see \cite{khar1}, \cite{opac}, and \cite{Fridman20013}), difference equations with
distributed delays (see the examples in \cite{Melchor-Aguilar2013}),
conservation laws modeled by first order hyperbolic partial differential
equations in which a transport phenomenon occurs, \cite{opac}, and other
classes of linear systems with distributed parameters, which have been shown
in \cite{Serth107} to admit a representation in the form of difference
equations in continuous time.

In the past decade (see \cite{Kharitonov200315}, \cite{khar1} and the reference therein),
the approach of Lyapunov-Krasovskii functionals with prescribed derivative,
defined by the so-called delay Lyapunov matrix, has shown its efficacy in a
wide range of applications for retarded, neutral and distributed delays
differential systems. Its extension to difference equations in
continuous time is a promising topic of investigation.

The purpose of this contribution is twofold. First, to set basic concepts and definitions of
the Lyapunov-Krasovskii approach aiming at this extension, in particular
the definition of the delay Lyapunov matrix for delay difference equations. Second, to provide a method for its construction, a step that is indeed essential in applications of the theory.

The paper is organized as follows: we introduce in Section II the class of difference equations under study and its fundamental matrix. In section III, we define the Lyapunov matrix $U(\tau )$, and we prove some of its
properties, as well as some properties of the jump discontinuities of its derivative. In section IV,
these properties allow us to present the analytic construction of the
Lyapunov matrix and of its derivative in the case of single and multiple
commensurate delays, and to propose an approximation in the non commensurate
case. Some illustrative examples are given for each case. The paper ends with
some concluding remarks.

\emph{Notations.} The transpose of a matrix $P$ is denoted by $P^T$, while
the smallest and the largest eigenvalues of a symmetric matrix $P$ are
denoted by $\lambda_{\min}(P)$ and $\lambda_{\max}(P)$, respectively. The
standard notation $P\succ 0$ ($P \succeq 0$) means that $P$ is a symmetric
positive definite matrix (semidefinite matrix). The space of piecewise right-continuous and bounded
functions defined on $[-h_m,0)$ is $\mathcal{PC}([-h_m,0),\mathbb{R}^n)$.
This space is endowed with the norm $\lVert \varphi\rVert_h= \sup_{-h_m\leq
\theta <0}\lVert \varphi(\theta)\rVert$ or with the $L_2$ norm $\lVert
\varphi\rVert_{L_2}^2=\int_{-h_m}^0\lVert \varphi(\theta)\rVert^2\mbox{d}%
\theta$, where $\lVert \varphi(\theta)\rVert$ stands for the Euclidean norm.
The solution at time $t$ of the system with initial condition $\varphi$ is
denoted by $x(t,\varphi)$, and $x_t(\varphi)=\{x(t+\theta,\varphi) \mbox{ }|%
\mbox{ } \theta \in [-h_m,0)\}$. If it is clear from the context, the
dependency with respect to $\varphi$ may be dropped. 

\section{Preliminaries}
Consider continuous-time difference equations of the form

\begin{equation}
x(t)=\sum_{j=1}^{m}A_{j}x(t-h_{j}),\quad t\geq 0  \label{eq:sys}
\end{equation}%
where $A_{1},\dots ,A_{m}$ are constant real $n\times n$ matrices, and $%
0=h_{0}<h_{1}\dots <h_{m}=H$ are the delays. For any piecewise
right-continuous and bounded initial function $\varphi \in \mathcal{PC}%
([-h_{m},0),\mathbb{R}^{n})$, there exists, for all $t\geq 0$, a unique
piecewise right-continuous and bounded solution $x(t,\varphi )$ of (\ref%
{eq:sys}), also referred to as the response of (\ref%
{eq:sys}). The Cauchy formula for the solutions of system (\ref{eq:sys}) is introduced next.

\begin{Teo}
For any initial function $\varphi \in \mathcal{PC}%
([-h_{m},0),\mathbb{R}^{n})$, the response of system (\ref%
{eq:sys}) for $t\geq 0$  is given by 
\begin{equation}  \label{eq:cau}
x(t,\varphi)=\sum_{j=1}^m\int_{-h_j}^0\frac{\mbox{{d}}}{\mbox{{d}}t%
}K(t-\theta-h_j)A_j\varphi(\theta)\mbox{\emph{d}}\theta
\end{equation}
where the $n\times n$ matrix function $K(t)$ satisfies
\begin{equation}  \label{eq:con2a}
K(t)=\sum_{j=1}^m K(t-h_j) A_j, \quad t\geq 0
\end{equation}%
\newline
with the initial condition 
\begin{equation}  \label{eq:K0}
K(\theta)=K_0:=\left(\sum_{j=1}^m A_j-I\right)^{-1}, \quad \theta\in[-h_m,0),
\end{equation}
assuming $\mbox{det}\left(I-\sum_{j=1}^m
A_j \mbox{\emph{e}}^{-\lambda h_j}\right)\neq 0$ for $\lambda=0$.
\end{Teo}

\begin{proof}
Let us consider the following identity: 
\begin{equation}  \label{eq:id1}
\int_{0}^{t} K(t-\theta)x(\theta)\mbox{d}\theta=\int_{0}^{t}
K(t-\theta)x(\theta)\mbox{d}\theta.
\end{equation}
Using (\ref{eq:con2a}) on the left-hand side (l.h.s.) of (\ref{eq:id1}) and (\ref{eq:sys}) on the right-hand side (r.h.s.), gives after some manipulations 
\begin{equation*}
\begin{aligned}
\sum_{j=1}^m\int_{t-h_j}^{t}&K(t-\theta-h_j)A_jx(\theta)\mbox{d}\theta=\\&=\sum_{j=1}^m\int_{-h_j}^{0}K(t-%
\theta-h_j)A_jx(\theta)\mbox{d}\theta. \end{aligned}
\end{equation*}
For $\theta \in (t-h_j,t]$, the expression $t-\theta-h_j \in
[-h_j,0)$, hence $K(t-\theta-h_j)=K_0$. Moreover, $x(\theta)=\varphi(%
\theta)$ on $\theta \in [-h_j,0)$. Thus, the preceding equation can be
rewritten as 
\begin{equation}  \label{eq:alt}
\sum_{j=1}^mK_0A_j\int_{t-h_j}^{t}x(\theta)\mbox{d}\theta=\sum_{j=1}^m\int_{-h_j}^{0}K(t-%
\theta-h_j)A_j\varphi(\theta)\mbox{d}\theta.
\end{equation}
Taking the first derivative with respect
to $t$ on both sides of (\ref{eq:alt}) we obtain

\[\begin{aligned}
K_0\sum_{j=1}^m A_j&\left(x(t)-x(t-h_j)\right)=
\\&= \sum_{j=1}^m\frac{\mbox{d}}{\mbox{d}t}\int_{-h_j}^{0}K(t-%
\theta-h_j)A_j\varphi(\theta)\mbox{d}\theta.
\end{aligned}
\]

We arrive at the desired result by applying the definition of $K_0$ in (\ref{eq:K0}) on the l.h.s. of the preceding equation, and because of the fact that the solution is right-continuous. %
\end{proof}

\begin{Cor}
Matrix $K(t)$ satisfies also the equation

\begin{equation}  \label{eq:con2b}
K(t)=\sum_{j=1}^m A_jK(t-h_j), \quad t\geq 0,
\end{equation}
with the same initial condition (\ref{eq:K0}).
\end{Cor}

\begin{proof}
We introduce the matrix $Q(t)$, unique solution
of 
\begin{equation}  \label{eq:eqpa}
Q(t)=\sum_{j=1}^mA_jQ(t-h_j), \quad t\geq 0
\end{equation}
\begin{equation*}
Q(\theta)=K_0, \quad \theta\in[-h_m,0).
\end{equation*}
Consider the following identity: 
\begin{equation}  \label{eq:iden1}
\int_{0}^{t}K(t-\theta)Q(\theta)\mbox{d}\theta=\int_{0}^{t}K(t-s)Q(s)%
\mbox{d}s.
\end{equation}

Replacing $Q(\theta)$ on the l.h.s. of (\ref{eq:iden1}) with (\ref{eq:eqpa}%
), and $K(t-s)$ on the r.h.s. with (\ref{eq:con2a}), we
have 
\begin{equation*}
\begin{aligned}
\sum_{j=1}^m\int_{0}^{t}K(t-\theta)A_j&Q(\theta-h_j)\mbox{d}\theta=\\&=\sum_{j=1}^m\int_{0}^{t}K(t-s-h_j)A_jQ(s)%
\mbox{d}s.\end{aligned}
\end{equation*}

The change of variable $s=\theta-h_j$ on the l.h.s. and algebraic manipulations yield

\begin{equation*}
\begin{aligned}
\sum_{j=1}^m\int_{-h_j}^{0}&K(t-s-h_j)A_jQ(s)\mbox{d}s=\\&=\sum_{j=1}^m\int_{t-h_j}^{t}K(t-s-h_j)A_jQ(s)\mbox{d}s. \end{aligned}
\end{equation*}

Notice that for $s\in (t-h_j,t]$, $K(t-s-h_j)=K_0$, and for $s \in [-h_j,0)$, $%
Q(s)=K_0$. Then, we arrive at 

\begin{equation*}
\sum_{j=1}\int_{t-h_j}^{t} K(\theta)\mbox{d}\theta A_jK_0=K_0\sum_{j=1}^m A_j\int_{t-h_j}^{t} Q(s)\mbox{d}s.
\end{equation*}

Taking the time derivative on both sides, we obtain
\[\begin{aligned}
\sum_{j=1}^m& \left(K(t)-K(t-h_j)\right)A_jK_0=\\&=K_0\sum_{j=1}^mA_j\left(Q(t)-Q(t-h_j)\right)
\end{aligned}
\]

The dynamic equations of $K(t)$ and $Q(t)$, and the definition of $K_0$ imply (\ref{eq:con2b}).
\end{proof}

Matrix $K(t)$ is known as the \emph{fundamental matrix} of system (\ref%
{eq:sys}) in reference to the fundamental matrix defined in \cite{Cooke1968}.




Equation (\ref{eq:con2b}) implies that each column of $K(t)$ is solution of (\ref{eq:sys}), therefore if the system is
exponentially stable, i.e. there exist $\sigma >0$ and $\gamma \geq 0$ such that

\[\lVert x(t,\varphi)\rVert \leq \gamma \lVert \varphi\rVert_h \mbox{e}^{-\sigma t},\]

the matrix $K(t)$ also satisfies the inequality%

\begin{equation}
\lVert K(t)\rVert \leq \gamma \lVert K_{0}\rVert \mbox{e}^{-\sigma
t},\quad t\geq 0.  \label{eq:222}
\end{equation}

\section{Lyapunov matrix and its properties}

In this section, we introduce the definition for the Lyapunov delay matrix.
We show that it is well defined, and we prove some of
it properties.

\begin{Lemma}
Let (\ref{eq:sys}) be exponentially stable, then for every $n\times n$
symmetric, positive definite matrix $W$, the matrix

\begin{equation}  \label{eq:U}
U(\tau)=\int_{0}^{\infty}\left(K(t)-K_0\right)^TWK(t+\tau)\mbox{\emph{d}}t
\end{equation}
is well defined for all $\tau \geq -H$.
\end{Lemma}

\begin{proof}
It follows directly from (\ref{eq:222}) that, for $t
\geq 0$ 

\begin{multline*}
\left\lVert \left(K^T(t)-K_0^T\right)WK(t+\tau)\right\rVert=\\=\left\lVert
\left(K(t)-K_0\right)^TWK(t+\tau)\right\rVert\\ \leq
\gamma\lVert K_0^TWK_0\rVert \left(\gamma\mbox{e}^{-\sigma(2t+\tau)}+
\mbox{e}^{-\sigma (t+\tau)}\right).\end{multline*}

Now, let $\tau \in [\tau_0, \infty)$; then, the inequality 
\begin{equation*}
\begin{aligned}
\int_{0}^{\infty}\lVert (K^T(t)-K_0^T)WK(t+\tau)\rVert\mbox{d}t\leq &\frac{%
\gamma^2}{2\sigma}\lVert K_0^TWK_0\rVert \mbox{e}^{-\sigma \tau_0}\\&\times\left(1+%
\frac{2}{\gamma}\right),
\end{aligned}
\end{equation*}
proves the statement.
\end{proof}

In analogy with the delay-free case, and other cases of delay systems
reported in the literature, the real valued matrix function $U(\cdot)$ is called the
Lyapunov matrix of (\ref{eq:sys}).

We now present some useful properties of $U(\tau)$.\\

\begin{Lemma}
Let system (\ref{eq:sys}) be exponentially stable. We define the $n \times n$
antisymmetric matrix 
\begin{equation}  \label{eq:P}
P \triangleq \int_{0}^{\infty}K^T(\tau)WK_0\mbox{\emph{d}}%
\tau-\int_{0}^{\infty}K_0^TWK(\tau)\mbox{\emph{d}}\tau
\end{equation}
with $K_0$ defined in (\ref{eq:K0}). Then, the Lyapunov matrix (\ref{eq:U})
associated to the symmetric positive definite matrix $W$ satisfies the Symmetry property: 

\begin{equation}  \label{eq:symU1}
\begin{aligned} U(-\tau)=U^T(\tau)+P-\tau K_0^TWK_0, \quad \tau \in [-H,H]
\end{aligned}
\end{equation}

and the Dynamic property: 

\begin{equation}  \label{eq:DYN11}
U(\tau)= \sum_{j=1}^m U(\tau-h_j)A_j, \quad \tau \geq 0.
\end{equation}%
\end{Lemma}

\begin{proof}
\textit{Symmetry property.}
Using the change of variable $\xi=t-\tau$ 
into (\ref{eq:U}) yields 
\begin{equation*}
\begin{aligned}
U(-\tau)&=&\int_{-\tau}^{\infty}\left(K^T(\xi+\tau)-K_0^T\right)WK(\xi)\mbox{d}\xi\\
&=&\int_{-\tau}^{\infty}K^T(\xi+\tau)W\left(K(\xi)-K_0\right)%
\mbox{d}\xi\\
&&-\int_{-\tau}^{\infty}K_0^TWK(\xi)\mbox{d}\xi+\int_{0}^{\infty}K^T(\xi)WK_0%
\mbox{d}\xi\\
&=&U^T(\tau)-\int_{-\tau}^{0}K_0^TWK(\xi)\mbox{d}\xi\\&&+%
\int_{0}^{\infty}K^T(\xi)WK_0\mbox{d}\xi-\int_{0}^{\infty} K_0^TWK(\xi)\mbox{d}\xi. \end{aligned}
\end{equation*}

Consider the case $\tau \geq 0$. Since the matrix $K(\xi)=K_0$, for $\xi\in
[-\tau,0)$, using the definition of matrix $P$ in (\ref{eq:P}) %
%
we arrive at equation (%
\ref{eq:symU1}) for $\tau \in [0,H)$.

Consider now the case $\tau \in [-H,0)$, for which the equality 
\begin{equation*}
U(\tau)=U^T(-\tau)+P+\int_{-0}^{\tau-0}K_0^TWK(\xi)\mbox{d}\xi,
\end{equation*}
is satisfied, and (\ref%
{eq:symU1}) follows by transposition.\newline

\textit{Dynamic property.} Using (\ref{eq:con2a}) into (\ref{eq:U}) gives 
\begin{equation*}
\begin{aligned} U(\tau)&=\int_{0}^{\infty}\left(K^T(t)-K_0\right)W\left(\sum_{j=1}^mK(t+\tau-h_j)A_j\right)%
 \mbox{d}t\\
 &=\sum_{j=1}^m\int_{0}^{\infty}\left(K^T(t)-K_0\right)WK(t+\tau-h_j)\mbox{d}tA_j.\\
\end{aligned}
\end{equation*}

Using the definition of $U(\tau)$, we arrive
at (\ref{eq:DYN11}).
\end{proof}\\


\begin{Lemma}
The matrix $P$ defined in (\ref{eq:P}) satisfies the equation 
\begin{equation}  \label{eq:defP}
P=K_0^T\left[\sum_{j=1}^mh_j\left(WK_0A_j-A_j^TK_0^TW\right)\right]K_0,
\end{equation}
\end{Lemma}

\begin{proof}
We begin by taking the Laplace transform of equation (\ref{eq:con2a}),
\begin{equation*}
\begin{aligned}
\hat{K}(s)&=\sum_{j=1}^m\int_{0}^{\infty}K(t-h_j)\mbox{e}^{-st}%
\mbox{d}tA_j
\\&=\sum_{j=1}^m\hat{K}(s)\mbox{e}^{-sh_j}A_j+\sum_{j=1}^m%
\int_{-h_j}^{0}K(t)\mbox{e}^{-s(t+h_j)}\mbox{d}tA_j\\
&=\sum_{j=1}^m\hat{K}(s)\mbox{e}^{-sh_j}A_j+\frac{1}{s}K_0\sum_{j=1}^m%
\left(1-\mbox{e}^{-sh_j}\right)A_j, \end{aligned}
\end{equation*}
which is equivalent to 
\begin{equation}  \label{eq:es1}
\begin{aligned}
\hat{K}(s)&=\frac{1}{s}K_0\sum_{j=1}^m\left(1-\mbox{e}^{-sh_j}%
\right)A_j\left(I-\sum_{j=1}^mA_j\mbox{e}^{-sh_j}\right)^{-1}. \end{aligned}
\end{equation}
Note that 
\begin{equation*}
\begin{aligned} K_0\sum_{j=1}^mA_j&=K_0\left(\sum_{j=1}^mA_j-I+I\right)=I+K_0. \end{aligned}
\end{equation*}
Using the definition of $K_0$ in (\ref{eq:K0}), we rewrite
(\ref{eq:es1}) as 
\begin{equation*}
\begin{aligned} \hat{K}(s)&=\frac{1}{s}
\left(I+K_0(I-\sum_{j=1}^m\mbox{e}^{-sh_j}A_j)\right)(I-%
\sum_{j=1}^mA_j\mbox{e}^{-sh_j})^{-1}, \end{aligned}
\end{equation*}
which yields 
\begin{equation*}
\begin{aligned}
\hat{K}(s)&=\frac{1}{s}\left(K_0-\left(\sum_{j=1}^mA_j%
\mbox{e}^{-sh_j}-I\right)^{-1}\right). \end{aligned}
\end{equation*}
Now, we define the matrix function 
\begin{equation*}
R(t)=-\int_{0}^tK_0^TWK(\tau)\mbox{d}\tau,
\end{equation*}
which has the following Laplace Transform

\begin{equation*}
\hat{R}(s)=\frac{1}{s^2}K_0^TW\left(\left(\sum_{j=1}^mA_j\mbox{e}%
^{-sh_j}-I\right)^{-1}-K_0\right).
\end{equation*}
From the definition of $P$ in (\ref{eq:P}), and the final value theorem, we
find that 
\begin{equation*}
P=\lim_{s\rightarrow 0} \left\{s\hat{R}(s)-s\hat{R}^T(s)\right\}.
\end{equation*}
It can be readily verified that 
\begin{equation*}
\begin{aligned} P=&\lim_{s\rightarrow 0}
\frac{1}{s}K_0^TW\left(\sum_{j=1}^mA_j\mbox{e}^{-sh_j}-I\right)^{-1}\\
&-\lim_{s\rightarrow 0}
\frac{1}{s}\left(\sum_{j=1}^mA_j^T\mbox{e}^{-sh_j}-I\right)^{-1}WK_0.
\end{aligned}
\end{equation*}
The series expansion of the term $\left(\sum_{j=1}^mA_j\mbox{e}%
^{-sh_j}-I\right)^{-1}$ allows to conclude that 
\begin{equation*}
\begin{aligned} \lim_{s \rightarrow
0}\frac{1}{s}&\left(\sum_{j=1}^mA_j\mbox{e}^{-sh_j}-I\right)^{-1}=\\&=\lim_{s
\rightarrow 0}\frac{1}{s}\left(I+sK_0\sum_{j=1}^mh_jA_j\right)K_0,
\end{aligned}
\end{equation*}
hence, 
\begin{equation*}
\begin{aligned} P=&\lim_{s\rightarrow 0}
\frac{1}{s}K_0^TW\left(I+sK_0\sum_{j=1}^mh_jA_j\right)K_0\\
&-\lim_{s\rightarrow 0}
\frac{1}{s}K_0^T\left(I+sK_0\sum_{j=1}^mh_jA_j\right)^TWK_0, \end{aligned}
\end{equation*}
and equation (\ref{eq:defP}) follows.
\end{proof}

Next, we introduce the matrix function $\Delta U'(\tau)$ describing the jump discontinuities of the derivative of the Lyapunov matrix. This matrix is defined as

\[ \Delta U'(\tau)\triangleq U'(\tau+0)-U'(\tau-0),
\]
and we show its following properties.\\

\begin{Lemma}
\label{le:delu} Let system (\ref{eq:sys}) be exponentially stable. The jump discontinuities of the derivative of the Lyapunov matrix (\ref{eq:U}),
associated to a positive definite matrix $W$ satisfy the Symmetry property:
\begin{equation}  \label{eq:symU}
\begin{aligned} \Delta U'(-\tau)=[\Delta U'(\tau)]^T, \end{aligned}
\end{equation}

the Dynamic property:
\begin{align}  \label{eq:DYN1}
\Delta U^{\prime }(\tau)=%
\begin{cases}
\sum_{j=1}^m \Delta U^{\prime }(\tau-h_j)A_j, & \tau > 0, \\ 
\sum_{j=1}^mA_j^T \Delta U^{\prime }(\tau+h_j), & \tau < 0,%
\end{cases}%
\end{align}
and the Generalized algebraic property:
\begin{equation}  \label{eq:alg1}
\begin{aligned} \sum_{i=1}^m\sum_{j=1}^mA_i^T\Delta
U'(\tau+h_i-h_j)A_j-\Delta U'(\tau)=& W\Delta K(\tau),\\ & \tau \geq 0
\end{aligned}
\end{equation}
\end{Lemma}

\begin{proof}
 \textit{Symmetry property.} From the definition of the Lyapunov matrix (\ref{eq:U}) we have, 
\begin{equation*}
\begin{aligned}
U'(\tau)&=\int_{0}^{\infty}\left(K(t)-K_0\right)^TWK'(t+\tau)\mbox{d}t,\qquad \tau \in
\mathbb{R}, \end{aligned}
\end{equation*}
which is a right-continuous function.

The fundamental matrix $K(t)$ is a constant function except at discontinuity
points depending on the delays $h_j$, $j=1,\dots,m$. Defining the set of
discontinuity instants of $K(t)$ as $\mathscr{I_K}=\{t_{\kappa}\}_{\kappa
\in \mathbb{N}}$, where 
\begin{equation}  \label{eq:tk}
t_{\kappa}\triangleq\min_{p_{\kappa}^1,\dots,p_{\kappa}^m}\left\{%
\sum_{j=1}^mp_{\kappa}^jh_j \mid t_{\kappa}>t_{{\kappa}-1}, \quad
p_{\kappa}^j \in \mathbb{N}\right\},
\end{equation}
we can write $U^{\prime }(\tau)$ as 
\begin{equation*}
\begin{aligned} U'(\tau)&=\sum_{\kappa\geq
0}\int_{t_{\kappa}-\tau-0}^{t_{\kappa}-\tau+0}\left(K^T(t)-K_0^T\right)WK'(t+\tau)\mbox{d}t,
\end{aligned}
\end{equation*}
which yields 
\begin{equation}  \label{eq:DUmat}
\begin{aligned} U'(\tau)&=\sum_{\kappa\geq 0}\left(K^T(t_{\kappa}-\tau)-K_0^T\right)W\Delta
K(t_{\kappa}), \end{aligned}
\end{equation}
where 
\begin{equation}  \label{eq:delk}
\Delta K(t)=K(t+0)-K(t-0), \quad t\in \mathbb{R}.
\end{equation}

We obtain for $\Delta U^{\prime }(\tau)=U^{\prime }(\tau+0)-U^{\prime
}(\tau-0) $, $\tau \in \mathbb{R}$ 
\begin{equation}  \label{eq:pDDU}
\Delta U^{\prime }(\tau)=-\sum_{\kappa\geq 0}\Delta
K^T(t_{\kappa}-\tau)W\Delta K(t_{\kappa}).
\end{equation}

From the definition of $\Delta K(t)$, we have that at least one of the terms
of the previous sum is not zero only when $t_{\kappa}-\tau \in \mathscr{I_K}$
for some $\kappa \in \mathbb{N}_0$. Otherwise, $\Delta U^{\prime }(\tau)=0$.
We define the set of values of $\bar{\tau} \in [a,b]$ such that for at least
one $\kappa \in \mathbb{N}_0$, $t_{\kappa}-\bar{\tau} \in \mathscr{I_K}$ as 
\begin{equation}  \label{eq:taub}
\mathscr{I_U}_{[a,b]}=\{\bar{\tau}\}_{\bar{\tau} \in [a,b]}.
\end{equation}

Defining the variable $t_q=t_{\kappa}-\bar{\tau}$, and 
\begin{equation*}
q(\bar{\tau})=\left\{q\in \mathbb{N} \mbox{ }|\mbox{ } t_q=%
\begin{cases}
\bar{\tau}, & \bar{\tau}\geq 0, \\ 
0, & \bar{\tau}< 0%
\end{cases}
\right\},
\end{equation*}%
\newline
we obtain 
\begin{equation*}
\Delta U^{\prime }(\bar{\tau})=-\sum_{q\geq q(-\bar{\tau})}\Delta
K^T(t_q)W\Delta K(t_q+\bar{\tau}).
\end{equation*}

If $\bar{\tau} \geq 0$, then $q(-\bar{\tau})=0$ and 
\begin{equation}  \label{eq:DDU}
\Delta U^{\prime }(\bar{\tau})=-\sum_{q\geq 0}\Delta K^T(t_q)W\Delta K(t_q+%
\bar{\tau}),
\end{equation}
otherwise, we can write 
\begin{equation*}
\begin{aligned} \Delta U'(\bar{\tau})=&-\sum_{q\geq 0}\Delta K^T(t_q)W\Delta
K(t_q+\bar{\tau})\\&+\sum_{0\leq q <q(-\bar{\tau})}\Delta K^T(t_q)W\Delta
K(t_q+\bar{\tau}). \end{aligned}
\end{equation*}
As $t_q+\bar{\tau}<0$, for $0\leq q <q(-\bar{\tau})$; then, $\Delta K(t_q+%
\bar{\tau})=0$, and (\ref{eq:DDU}) holds.

Given the two definitions of $\Delta U^{\prime }(\tau)$ that we have found
in (\ref{eq:pDDU}) and (\ref{eq:DDU}), it is straightforward to verify that
property (\ref{eq:symU}) is satisfied in the case $t_q-\tau \in \mathscr{I_K}
$, for at least one $q \in \mathbb{N}$. Otherwise, $t_q-\tau \notin %
\mathscr{I_K}$, $t_q+\tau \notin \mathscr{I_K}$ for all $q \in \mathbb{N}$,
and $\Delta U^{\prime }(\tau)=0=\left(\Delta U^{\prime }(-\tau)\right)^T$. \\

 \textit{Dynamic property.} From the definition of the fundamental matrix in (\ref{eq:con2a}), we have
\begin{equation*}
K(t)=%
\begin{cases}
\sum_{j=1}^m K(t-h_j)A_j, & t\geq 0 \\ 
K_0, & t<0,%
\end{cases}%
\end{equation*}%
\newline
therefore, for $t>0$, (\ref{eq:delk}) yields 
\begin{equation*}
\begin{aligned} \Delta K(t)&=
\sum_{j=1}^m\left(K(t+0-h_j)-K(t-0-h_j)\right)A_j,\\ &=\sum_{j=1}^m\Delta
K(t-h_j)A_j, \end{aligned}
\end{equation*}
and for $t=0$, $\Delta K(0)=I$.
Summarizing these results, we have 
\begin{equation}  \label{eq:DK}
\Delta K(t)=%
\begin{cases}
\sum_{j=1}^m \Delta K(t-h_j)A_j, & t> 0 \\ 
 
I, & t=0, \\ 

0, & t <0.%
\end{cases}%
\end{equation}

Considering the definition of $\Delta U^{\prime }(\tau)$ in (\ref{eq:DDU})
we can write 
\begin{equation*}
\begin{aligned} \Delta U'(\tau) &= - \sum_{\kappa \geq 0} \Delta
K^T(t_{\kappa})W\Delta K(t_{\kappa}+\tau).\\ \end{aligned}
\end{equation*}

In view of the dynamics of $\Delta K(t)$ given in (\ref{eq:DK}), it follows
that for $\tau > 0$, 
\begin{equation}  \label{eq:dynpos}
\begin{aligned} \Delta U'(\tau) &= - \sum_{\kappa \geq 0} \Delta
K^T(t_{\kappa})W \sum_{j=1}^m \Delta K (t_{\kappa}+\tau-h_j)A_j,& \\ &=
-\sum_{j=1}^m \sum_{\kappa \geq 0} \Delta K^T(t_{\kappa})W \Delta K
(t_{\kappa}+\tau-h_j)A_j,&\\ &= \sum_{j=1}^m \Delta U'(\tau-h_j)A_j,
\end{aligned}
\end{equation}
and (\ref{eq:DYN1}) is proved for $\tau>0$.
Consider $\tau<0$, so that $-\tau>0$, and $\Delta U^{\prime }(-\tau)$
satisfies equation (\ref{eq:dynpos}) as follows 
\begin{equation*}
\Delta U^{\prime }(-\tau)=\sum_{j=1}^m\Delta U^{\prime }(-\tau-h_j)A_j,
\quad \tau <0,
\end{equation*}
as the matrix $\Delta U^{\prime }(-\tau)$ satisfies the symmetry property (%
\ref{eq:symU}), we can apply it on both sides and obtain 
\begin{equation*}
\left(\Delta U^{\prime }(\tau)\right)^T=\sum_{j=1}^m\left(\Delta U^{\prime
}(\tau+h_j))\right)^TA_j, \quad \tau <0.
\end{equation*}
Transposition proves (\ref{eq:DYN1}) for $\tau<0$. 

\textit{\ Generalized algebraic property.} We will consider both terms on the l.h.s. of equation (\ref{eq:alg1}%
) separately. For $\Delta U^{\prime }(\tau)$, $\tau \geq 0$, defined by
equation (\ref{eq:DDU}), we can write 
\begin{equation*}
\begin{aligned} \Delta U'(\tau)=&-\sum_{\kappa > 0} \Delta K^T(t_{\kappa})W
\Delta K (t_{\kappa}+\tau)\\&-\Delta K^T(0)W\Delta K(\tau). \end{aligned}
\end{equation*}
Using the dynamics of $\Delta K(t)$ described in (\ref{eq:DK}), we
get
\begin{equation*}
\begin{aligned} \Delta U'(\tau)=&-\sum_{\kappa > 0}\sum_{i=1}^m A_i^T\Delta
K^T(t_{\kappa}-h_i)W\\&\times \sum_{j=1}^m\Delta K
(t_{\kappa}+\tau-h_j)A_j-W\Delta K(\tau), \end{aligned}
\end{equation*}
which is equal to 
\begin{equation}  \label{eq:apro1}
\begin{aligned} \Delta U'(\tau)=&-\sum_{\kappa > 0}\sum_{i=1}^m\sum_{j=1}^m
A_i^T\Delta K^T(t_{\kappa}-h_i)W\\ &\times \Delta K
(t_{\kappa}+\tau-h_j)A_j-W\Delta K(\tau). \end{aligned}
\end{equation}
Now, for $\sum_{i=1}^m\sum_{j=1}^m A_i^T\Delta U^{\prime }(\tau+h_i-h_j)A_j$%
, consider again equation (\ref{eq:DDU}) 
\begin{equation*}
\begin{aligned} \Delta U'(\tau+h_i-h_j)=&-\sum_{\kappa \geq 0}\Delta
K^T(t_{\kappa})W\\ &\times \Delta K (t_{\kappa}+\tau+h_i-h_j), \end{aligned}
\end{equation*}
the change of variable $t_q=t_{\kappa}+h_i$ allows us to write this equation
as 
\begin{equation*}
\begin{aligned} \Delta U'(\tau+h_i-h_j)&=&-\sum_{q \geq q(h_i)}\Delta
K^T(t_q-h_i)W\\ &&\times \Delta K (t_q+\tau-h_j)\\ &=&-\sum_{q \geq 0}\Delta
K^T(t_q-h_i)W\\&&\times \Delta K (t_q+\tau-h_j)\\&&+\sum_{0\leq
q<q(h_i)}\Delta K^T(t_q-h_i)W \\ &&\times \Delta K (t_q+\tau-h_j).
\end{aligned}
\end{equation*}
The last term is canceled given that $t_q-h_i<0$ for $0\leq q <q(h_i)$, and $%
\Delta K(\theta)=0$ for $\theta<0$. Finally, we arrive at 
\begin{equation}  \label{eq:DDUboth}
\begin{aligned} \Delta U'(\tau+h_i-h_j)=&-\sum_{q \geq 0}\Delta
K^T(t_q-h_i)W\\&\times \Delta K (t_q+\tau-h_j). \end{aligned}
\end{equation}
As a consequence, 
\begin{equation*}
\begin{aligned} &\sum_{i=1}^m\sum_{j=1}^m A_i^T\Delta
U'(\tau+h_i-h_j)A_j=\\=&-\sum_{q > 0}\sum_{i=1}^m\sum_{j=1}^m A_i^T\Delta
K^T(t_q-h_i)W \Delta K (t_q+\tau-h_j)A_j\\&-\sum_{i=1}^m\sum_{j=1}^m
A_i^T\Delta K^T(-h_i)W \Delta K (\tau-h_j)A_j \end{aligned}.
\end{equation*}
The last term is equal to zero as $\Delta
K(-h_i)=0 $, hence
\begin{equation}  \label{eq:apro2}
\begin{aligned} &\sum_{i=1}^m\sum_{j=1}^m A_i^T\Delta
U'(\tau+h_i-h_j)A_j=\\&= -\sum_{q > 0}\sum_{i=1}^m\sum_{j=1}^m A_i^T\Delta
K^T(t_q-h_i)W \Delta K (t_q+\tau-h_j)A_j, \end{aligned}
\end{equation}
Subtracting (\ref{eq:apro1}) from (\ref{eq:apro2}) proves (\ref{eq:alg1}).
\end{proof}

\section{Construction of the Lyapunov Matrix}\label{sec:con}
Any application of theoretical results requires an effective numerical procedure for constructing the matrices $U(\tau)$ and $\Delta U'(\tau)$. We present next the construction of matrix $U(\tau)$ that satisfies properties (\ref{eq:symU1})-(\ref{eq:DYN11}) and whose derivative's jump discontinuities satisfy (\ref{eq:symU})-(\ref{eq:alg1}).

\subsection{Single-Delay Case} \label{sec:onedel}

For the difference equation
\begin{equation} \label{eq:sysOD}
\begin{aligned}
x(t)&=Ax(t-H),\quad t\geq 0,\\
x(\theta)&= \varphi(\theta), \quad \theta \in [-H,0),
\end{aligned}
\end{equation}
with $x(t) \in  \mathbb{R}^n$, the fundamental matrix is given by
\begin{equation}\label{eq:fun1}
K(t)=K(t-H)A, \quad t\geq 0,
\end{equation}
with initial condition
\begin{equation}\label{eq:infunod}
K(\theta)=K_0=\left(A-I\right)^{-1}, \quad \theta \in [-H,0).
\end{equation}

%

For any positive definite matrix $W,$ the Lyapunov matrix (\ref{eq:U}) satisfies the properties (\ref{eq:symU1})-(\ref{eq:DYN11}). We consider the following equalities, given by the dynamic property (\ref{eq:DYN11}), for $\xi \in [0,H]$,
\begin{align*}
U(\xi)=&U(\xi-H)A,\\ 
U^T(H-\xi)=&A^TU^T(\xi).
\end{align*}
Applying the symmetry property (\ref{eq:symU1}) to the second equation we have
\begin{align}\label{eq:hpro39}
U(\xi)=&U(\xi-H)A,\\ 
U(\xi-H)=&A^TU(\xi)-\left(K_0^{-1}\right)^TP-\left(\xi I+HK_0^T\right)WK_0.\label{eq:hpro40}
\end{align}
where $P$ defined in (\ref{eq:defP}), is given by
\begin{align*}
P=HK_0^TWK_0^2-H(K_0^2)^TWK_0.
\end{align*}
We define the following variables, for $\xi \in [0,H]$,
\begin{equation}\label{eq:hyz}
\begin{aligned}
Y(\xi)&=U(\xi),\\
Z(\xi)&=U(\xi-H).
\end{aligned}
\end{equation}
By writing (\ref{eq:hpro39})-(\ref{eq:hpro40}) in terms of the variables defined in (\ref{eq:hyz}), the following system of equations is obtained
\begin{equation}\label{eq:hsyseq1}
\begin{aligned}
Y(\xi)-Z(\xi)A &=0\\
Z(\xi)- A^TY(\xi)&=-\left(K_0^{-1}\right)^TP-\left(\xi I+HK_0^T\right)WK_0.
\end{aligned}
\end{equation}

 The linear system (\ref{eq:hsyseq1}) is solved for matrices $Y(\xi)$ and $Z(\xi)$ by defining the vectors $y(\xi)=\mbox{vec}(Y(\xi))$ and $z(\xi)=\mbox{vec}(Z(\xi))$, and using Kronecker product properties: the system of equations (\ref{eq:hsyseq1}) is rewritten as
\begin{equation}\label{eq:hsyseq2}
\begin{aligned}
&\begin{bmatrix}
I\otimes I  &  -A^T\otimes I\\
-I\otimes A^T& I\otimes I
\end{bmatrix}\begin{bmatrix}
y(\xi)\\z(\xi)
\end{bmatrix}=\\&=\begin{bmatrix}
0_{n^2}\\ -\mbox{vec}(\left(K_0^{-1}\right)^TP+\left(\xi I+HK_0^T\right)WK_0)
\end{bmatrix}.
\end{aligned}
\end{equation}

This linear system has a unique solution (\ref{eq:hsyseq1}) if the matrix $I\otimes I-A^T\otimes A^T$
is invertible, that is, if none of the eigenvalues of $A$ lies on the unit circle of the complex plane.

Recalling the definitions in (\ref{eq:hyz}), we recover $u(\xi)=\mbox{vec}(U(\xi))$, $\xi \in [0,H]$, by devectorization of $y(\xi)$.

\begin{Examp}
In Fig. \ref{fig:hU2} we present the construction of $U(\tau)$, $\tau \in [-H,H]$ associated to a one-dimensional system, for a given positive definite matrix $W=I_2$, and a given parameter matrix $A=\begin{bmatrix}
-0.9375&1.11844\\0.3732&-1.3009
\end{bmatrix}$.
\newcommand{\mi}{\boldsymbol{-} \mathrel{\mkern -16mu} \boldsymbol{-}}

The entries of the matrix $U(\tau)$ appear as follows: 
$U_{11}(\tau)$ (\textcolor[rgb]{0.9,0.36,0.055}{$\mi$}), $U_{12}(\tau)$(\textcolor[rgb]{0.35,0.42,0.753}{\textbf{$\mi$}}), $U_{21}(\tau)$ (\textcolor[rgb]{0.945,0.58,0}{\textbf{$\mi$}}), and $U_{22}(\tau)$ (\textcolor[rgb]{0.65,0.25,0.69}{\textbf{$\mi$}}), in all the forthcoming figures.

\begin{figure}[H]
\centering
\includegraphics[height=35mm]{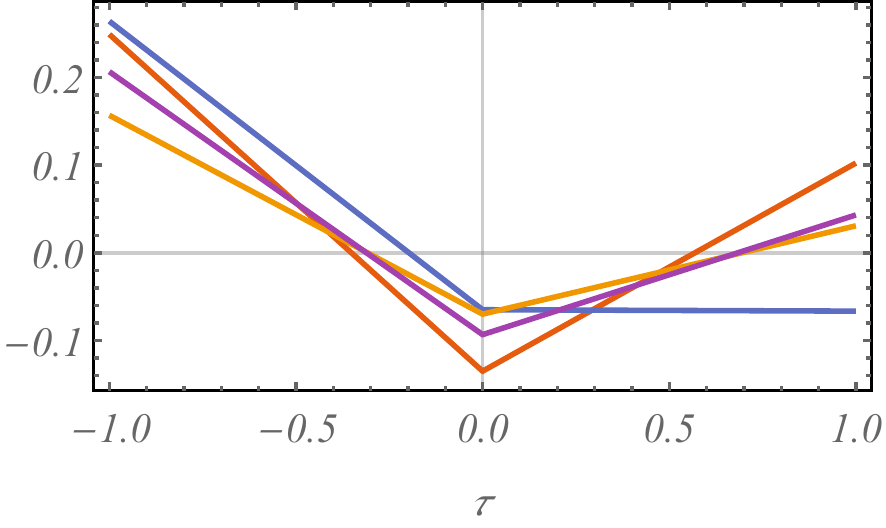}
\caption{Graph of $U(\tau)$, for a two-dimensional system of the form (\ref{eq:sysOD}). } \label{fig:hU2}
 \end{figure}

\end{Examp}

\subsection{Case of Multiple Commensurate Delays}\label{sec:muldel}

It is clear that the case of commensurate delays can be reduced to the one delay case, however, it seems important to address it as a multiple-delay system as a preparatory step to the case of multiple non commensurate delays.

For the case of difference equation in continuous time with multiple commensurate delays of the form
\begin{equation}\label{eq:mulcon}
\begin{aligned}
x(t)&=\sum_{j=1}^mA_jx(t-jh),\quad t\geq 0, \\
x(\theta)&= \varphi(\theta), \quad \theta \in [-mh,0),
\end{aligned}
\end{equation}
where $h$ is known as the basic delay, the fundamental matrix is given by
\begin{equation}\label{eq:fun10}
K(t)=\sum_{j=1}^mK(t-jh)A_j, \quad t\geq 0,
\end{equation}
with initial condition
\begin{equation}\label{eq:funinM}
K(\theta)=K_0=\left(\sum_{j=1}^mA_j-I\right)^{-1}, \quad \theta \in [-mh,0).
\end{equation}
We consider the following equalities, given by the dynamic  and the symmetry properties, (\ref{eq:DYN11}) and (\ref{eq:symU1}), respectively. For $\xi \in [0,h]$,
\begin{equation}\label{eq:dinamM}
\begin{aligned}
U(kh+\xi)=&\sum_{j=1}^mU\left((k-j)h+\xi\right)A_j, & k =0,\dots,m-1, \\ 
\end{aligned}
\end{equation}
\begin{equation}\label{eq:dinamM2}
\begin{aligned}
U(\xi-kh)=&\sum_{j=1}^mA_j^TU\left(\xi+(j-k)h\right)-\left(K_0^{-1}\right)^TP\\
&-\left((\xi-kh)I+\sum_{j=1}^{m}h_j A_j^TK_0^T\right)WK_0, \end{aligned}
\end{equation}
for $k =1,2,\dots,m.$, where $P$ is defined in (\ref{eq:defP}).
Let us define the auxiliary matrices, for $\xi \in [0,h]$
\begin{equation}\label{eq:aux1}
Y_k(\xi)=U(kh+\xi), \quad k\in\{-m, -m+1, \dots,0,\dots, m-1\}.
\end{equation}
In the new variables introduced in (\ref{eq:aux1}), the equations in (\ref{eq:dinamM}), for $k \in \{0,1,\dots,m-1\}$, are rewritten as
\begin{equation}\label{eq:sisteq10}
\begin{aligned}
Y_k(\xi)=\sum_{j=1}^m Y_{k-j}(\xi)A_j, \qquad k \in \{0,1,\dots,m-1\},
\end{aligned}
\end{equation}
and for $k =1,\dots,m$, the equations in (\ref{eq:dinamM2}) are
\begin{equation}\label{eq:sisteq12}
\begin{aligned}
Y_{-k}(\xi)=&\sum_{j=1}^m A_j^TY_{-k+j}(\xi)-\left(K_0^{-1}\right)^TP\\&-\left((\xi-kh)I+\sum_{j=1}^{m}h_j A_j^TK_0^T\right)\\&\times WK_0.
\end{aligned}
\end{equation}
Observe that (\ref{eq:sisteq10})-(\ref{eq:sisteq12}) is a system of $2m$ algebraic equations with $2m$ unknowns defined in (\ref{eq:aux1}). It can be rewritten in vector form using Kronecker products, and solved for the vectors $y_{k}(\xi)=\mbox{vec}(Y_{k}(\xi))$, $k \in \{-m,1-m,\dots,-1,0,1,\dots,m-1\}$, $\xi \in [0,h]$.\\

\begin{Cor}
If the system of equations (\ref{eq:sisteq10})-(\ref{eq:sisteq12}) admits a unique solution
\[
\left\{Y_{m-1}(\xi),Y_{m-2}(\xi),\dots,Y_0(\xi),\dots,Y_{-m}(\xi)\right\}, \quad \xi \in [0,h],
\]
then, there exists a unique Lyapunov matrix $U(\tau)$ associated to matrix $W.$ This matrix is defined on $[0,H]$ by 
\[
U(kh+\xi)=Y_k(\xi),\quad \xi \in [0,h], \quad k=0,1,\dots,m-1.
\]
\end{Cor}

\begin{Examp}\label{ex:ex1} 
\emph{Two-delay system.} 
Consider the following continuous-time difference equation
\begin{equation}\label{eq:comm1}
\begin{aligned}
x(t)&=A_1x(t-1)+A_2x(t-3/2), & t\geq 0 \\\\
x(\theta)&= \varphi(\theta), & \theta \in [-3/2,0),
\end{aligned}
\end{equation}
where $x(t) \in \mathbb{R}^n$ and the basic delay $h$ is equal to $1/2$. 
We solve for $U(\tau)$, $\tau \in[-3/2,3/2]$, by finding the solution of the system of equations (\ref{eq:sisteq10})-(\ref{eq:sisteq12}), using Kronecker products.
Plots of $U(\tau)\in \mathbb{R}^2$, $\tau \in [-3/2,3/2]$, are shown in Fig. \ref{fig:f11} for $W=I_2$, and different values of $A_1$ and $A_2$. 
Fig \ref{fig:f11} (a): $A_1=\begin{bmatrix}
-0.4 & -0.3\\ 0.1 &0.15
\end{bmatrix}$ ; $A_2=\begin{bmatrix}
0.1 & 0.25\\ -0.9 & -0.1
\end{bmatrix}$,\\
Fig \ref{fig:f11} (b): $A_1=\begin{bmatrix}
1.1 & 0\\ -0.4 &0 
\end{bmatrix}$ ; $A_2=\begin{bmatrix}
0.25 & -0.125\\ -0.4 &-0.5 
\end{bmatrix}$.

\newcommand{\mia}{\boldsymbol{-} \mathrel{\mkern -20mu} \boldsymbol{-}}
\begin{figure}[H]
\centering
\subfigure[]{\includegraphics[height=35mm]{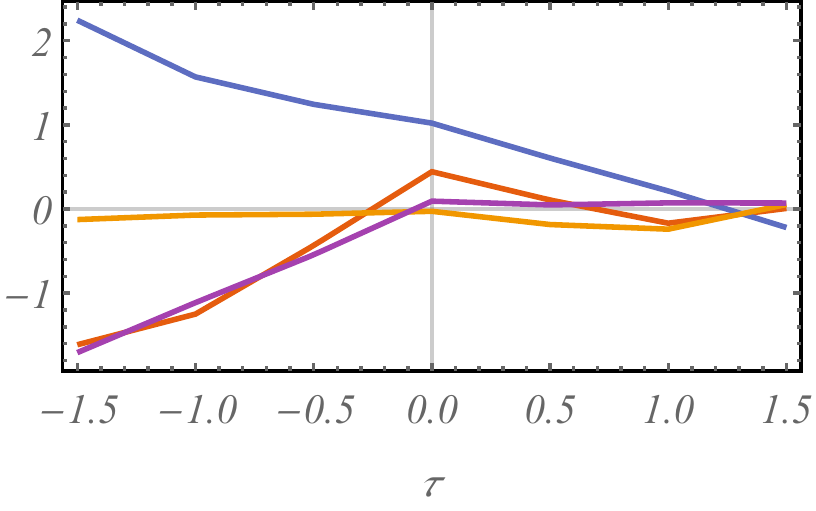}}\\
\subfigure[]{\includegraphics[height=35mm]{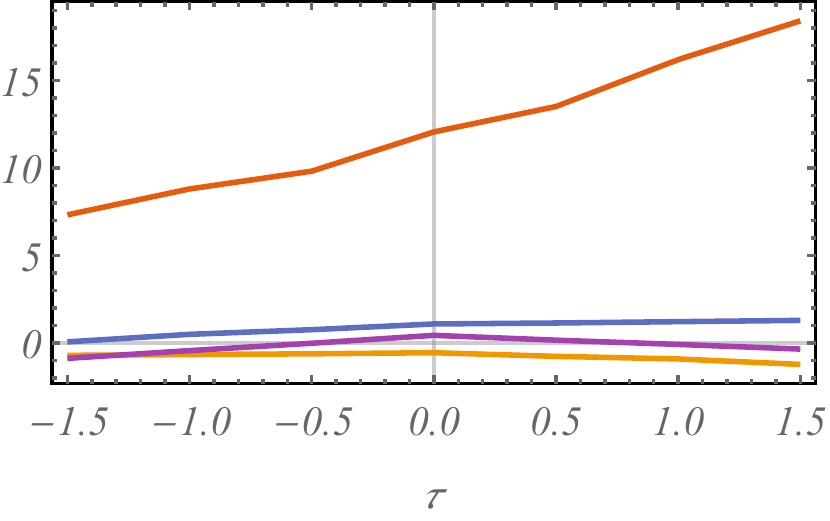}}\\
\caption{$U(\tau)$, $\tau \in [-3/2,3/2]$ related to the two-dimensional system of two commensurate delays (\ref{eq:comm1}). } \label{fig:f11}
 \end{figure}
\end{Examp}

\subsection{Case of Non-Commensurate Delays: Proposal of an Approximation }\label{sec:si}

In this section, we propose a strategy to approximate the matrix function $U(\tau)$ of a system of two non commensurate delays. We do so by computing the matrices $U_s(\tau)$ of a systems of commensurate delays parametrized by $s$ that tends to the non commensurate case. The intuition behind this approach is that a sufficiently good approximation obtained from the systems with commensurate delays is attainable, since it will become noticeable that the sequence of the matrix functions $U_s(\tau)$ tends to a continuous function.
%
%
Consider the system described by equation (\ref{eq:sys}) . We take the continued fraction representation (\cite{wall1948analytic}) of $h_i$, $i=\overline{1,m}$:
\begin{equation}\label{eq:cfe}
h_i=[\lambda_{i,0}; \lambda_{i,1}, \lambda_{i,2}, \dots]
\end{equation}

A rational approximation of $h_i$, $i=\overline{1,m}$ is given by a finite number of terms on the r.h.s. of equality (\ref{eq:cfe}), allowing us to find an approximation of $U(\tau)$ by using the strategy of the multiple commensurate delay case. The solution of the function $U_s(kh+\xi)$, $\xi \in [0,h]$, is obtained by solving the system of equations (\ref{eq:dinamM})-(\ref{eq:dinamM2}) for $k= 0,1,\dots,\frac{\bar{h}_m}{h}-1$,\\
with
$\bar{h}_i=[\lambda_{i,0};\lambda_{i,1},\lambda_{i,2},\cdots,\lambda_{i,s}], i=\overline{1,m}$\\
and 
$h=\gcd{(\bar{h}_1,\bar{h}_2,\dots,\bar{h}_m})$, as basic delay.\\ 

\begin{Examp}\label{ex:ex3} Consider the system described by
\begin{equation}\label{eq:cl11}
x(t)=\begin{bmatrix}
-0.4 & -0.3\\ 0.1+a & 0.15
\end{bmatrix}x(t-1)+\begin{bmatrix}
0.1 & 0.25\\ -0.9 & -0.1+b
\end{bmatrix}x(t-\sqrt{2}).
\end{equation}

The continued fraction representation of $\sqrt{2}$ is given by
\begin{equation}\label{eq:cfe2}
\sqrt{2}=[1; 2, 2, \dots]
\end{equation}
The scalar function $ U_s(\tau)$, $\tau \in [-\bar{h}_2,\bar{h}_2]$ is sketched for different values of $a$, $b$ and $s$. Fig. \ref{fig:fnc1} corresponds to $a=0.7$, $b=-1.1$ with (a) $s=1$ ( $\sqrt{2}$ is approximated by $\frac{3}{2}$), (b) $s=4$ ($\sqrt{2}$ is approximated by $\frac{41}{29}$). and (c)  $s=7$ ($\sqrt{2}$ is approximated by $\frac{577}{408}$). 
Fig. \ref{fig:fnc2} corresponds to $a=-1.6$, $b=-0.4$ with (a) $s=1$, (b)  $s=4$ and (c) $s=6$. 
%

\begin{figure}
\centering
\subfigure[]{\includegraphics[height=35mm]{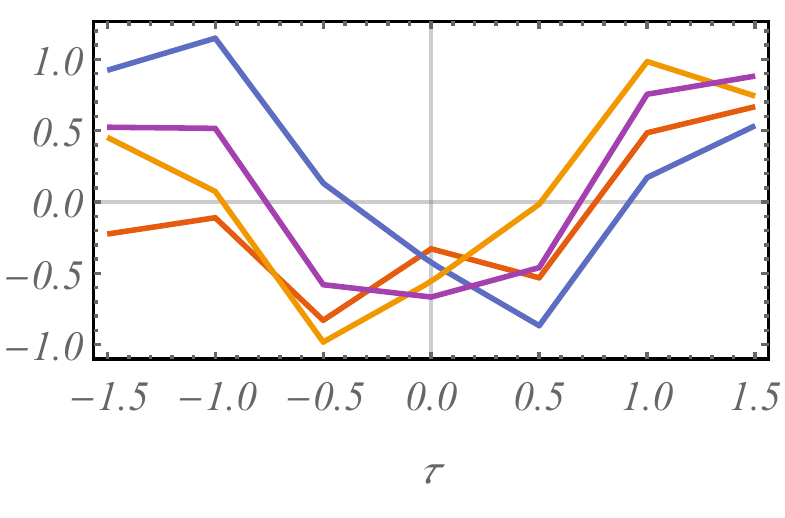}}\\
\subfigure[]{\includegraphics[height=35mm]{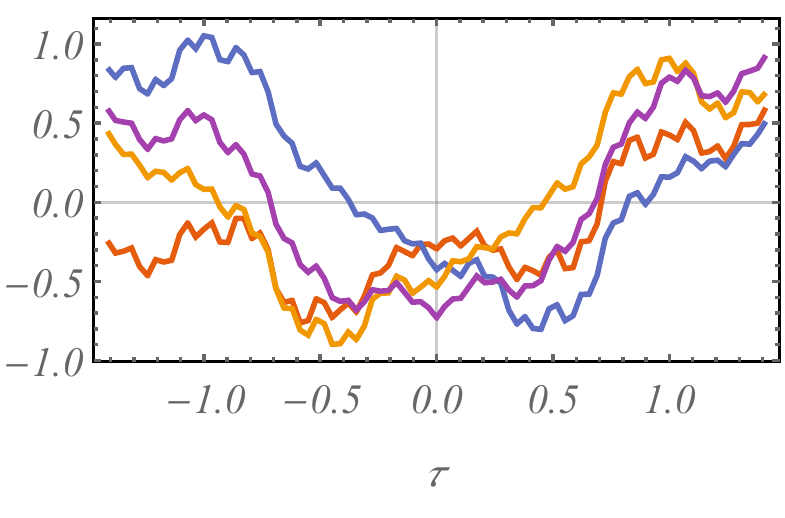}}\\
\subfigure[]{\includegraphics[height=35mm]{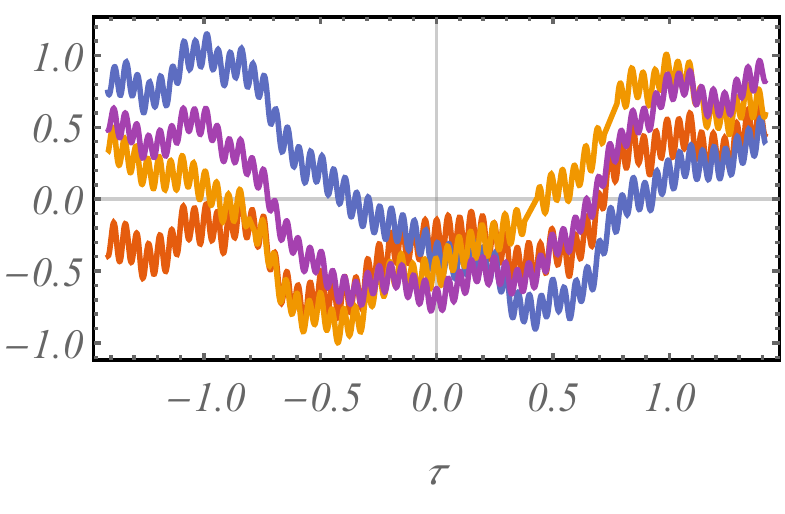}}\\
\caption{Approximation of $U(\tau)$, $\tau \in [-\bar{h}_2,\bar{h}_2]$ related to the two-dimensional system of two non commensurate delays (\ref{eq:cl11}), $a=0.7, b=-1.1$.} \label{fig:fnc1}
 \end{figure}
 
 \begin{figure}
\centering
\subfigure[]{\includegraphics[height=35mm]{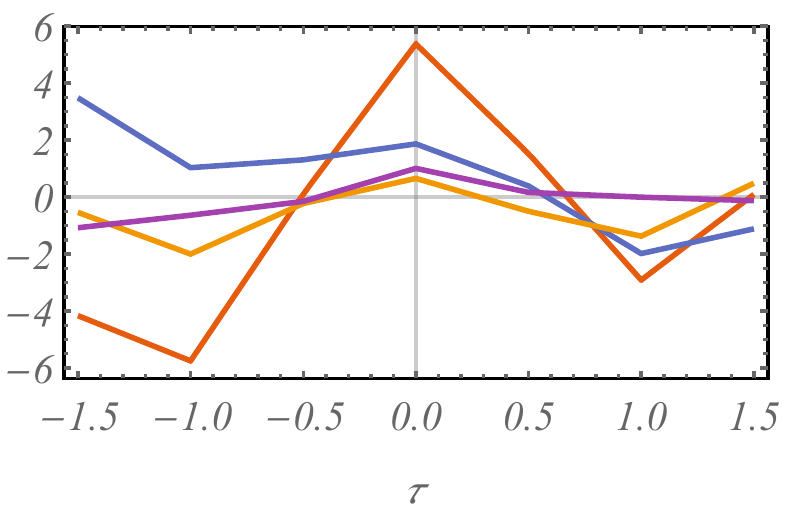}}\\
\subfigure[]{\includegraphics[height=35mm]{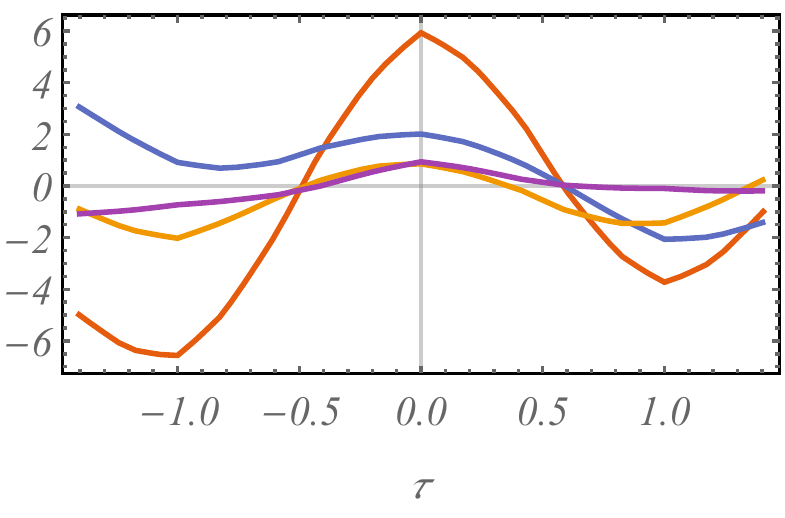}}\\
\subfigure[]{\includegraphics[height=35mm]{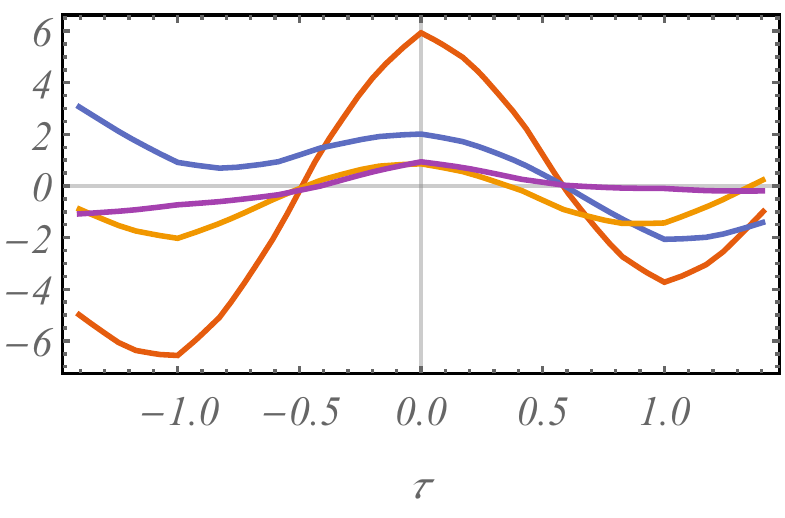}}\\
\caption{Approximation of $U(\tau)$, $\tau \in [-\bar{h}_2,\bar{h}_2]$ related to the two-dimensional system of two non commensurate delays (\ref{eq:cl11}), $a=-1.6, b=-1.4$.} \label{fig:fnc2}
 \end{figure}
\end{Examp}

\section{Conclusion}

The definition of the Lyapunov delay matrix of difference equations in
continuous time is introduced. Some of its properties and of the jump discontinuities of its derivative
are proved. These results allow the presentation of an analytic method for
  cases, and for its approximation in the non commensurate case. These results are key preliminary steps of our current
research on difference equations in continuous time, namely, the
construction of functionals with prescribed derivative, and the assessment of the system stability via the delay Lyapunov matrix $U(\tau )$, where the strategy introduced in \cite{emo15} is used.

\bibliography{bibliography/bibliography}             
                                                   







\end{document}